\documentclass[12pt,leqno]{article}
\usepackage{amsfonts}
\usepackage{amsmath, amsthm, amsfonts, amssymb, color}
\usepackage{mathrsfs}
\renewcommand{\bar}{\overline}
\renewcommand{\hat}{\widehat}
\renewcommand{\tilde}{\widetilde}

\newtheorem{thm}{Theorem}[section]

\newtheorem{lem}[thm]{Lemma}

\theoremstyle{definition}
\newtheorem{defn}{Definition}[section]
\newcommand{\scr}[1]{\mathscr #1}
\definecolor{wco}{rgb}{0.5,0.2,0.3}
\allowdisplaybreaks
\numberwithin{equation}{section} \theoremstyle{remark}
\newtheorem{rem}{Remark}[section]

\newcommand{\ua}{\uparrow}

\title{{\bf CLT and MDP for McKean-Vlasov SDEs }
}
\author{
{\bf  Yongqiang Suo and Chenggui Yuan
 }\\
\footnotesize{Department of Mathematics, Swansea University, Bay Campus, SA1 8EN, UK}\\
}

\begin{document}
\def\A{\mathscr{A}}
\def\G{\mathscr{G}}
\def\eq{\equation}
\def\bg{\begin}
\def\ep{\epsilon}
\def\111{±ßÖµÎÊÌâ(1)--(2)}
\def\x{\|x\|}
\def\y{\|y\|}
\def\xr{\|x\|_r}
\def\xrr{(\sum_{i=1}^T|x_i|^r)^{\frac{1}{r}}}
\def\R{\mathbb R}
\def\ff{\frac}
\def\ss{\sqrt}
\def\B{\mathbf B}
\def\N{\mathbb N}
\def\kk{\kappa} \def\m{{\bf m}}
\def\dd{\delta} \def\DD{\Dd} \def\vv{\varepsilon} \def\rr{\rho}
\def\<{\langle} \def\>{\rangle} \def\GG{\Gamma} \def\gg{\gamma}
  \def\nn{\nabla} \def\pp{\partial} \def\EE{\scr E}
\def\d{\text{\rm{d}}} \def\bb{\beta} \def\aa{\alpha} \def\D{\scr D}
  \def\si{\sigma} \def\ess{\text{\rm{ess}}}\def\lam{\lambda}
\def\beg{\begin} \def\beq{\begin{equation}}  \def\F{\scr F}
\def\Ric{\text{\rm{Ric}}} \def \Hess{\text{\rm{Hess}}}
\def\e{\text{\rm{e}}} \def\ua{\underline a} \def\OO{\Omega}  \def\oo{\omega}
 \def\tt{\tilde} \def\Ric{\text{\rm{Ric}}}
\def\cut{\text{\rm{cut}}} \def\P{\mathbb P} \def\ifn{I_n(f^{\bigotimes n})}
\def\C{\scr C}      \def\alphaa{\mathbf{r}}     \def\r{r}
\def\gap{\text{\rm{gap}}} \def\prr{\pi_{{\bf m},\varrho}}  \def\r{\mathbf r}
\def\Z{\mathbb Z} \def\vrr{\varrho} \def\l{\lambda}
\def\L{\scr L}\def\Tilde{\tilde} \def\TILDE{\tilde}\def\II{\mathbb I}
\def\i{{\rm in}}\def\Sect{{\rm Sect}}\def\E{\mathbb E} \def\H{\mathbb H}
\def\M{\scr M}\def\Q{\mathbb Q} \def\texto{\text{o}} \def\LL{\Lambda}
\def\Rank{{\rm Rank}} \def\B{\scr B} \def\i{{\rm i}} \def\HR{Hat{\R}^d}
\def\to{\rightarrow}\def\l{\ell}\def\ll{\lambda}
\def\8{\infty}\def\ee{\epsilon} \def\Y{\mathbb{Y}} \def\lf{\lfloor}
\def\rf{\rfloor}\def\3{\triangle}\def\H{\mathbb{H}}\def\S{\mathbb{S}}\def\1{\lesssim}
\def\va{\varphi}

\def\R{\mathbb R}  \def\B{\mathbf
B}
\def\N{\mathbb N} \def\kk{\kappa} \def\m{{\bf m}}
\def\dd{\delta} \def\DD{\Delta} \def\vv{\varepsilon} \def\rr{\rho}
\def\<{\langle} \def\>{\rangle} \def\GG{\Gamma} \def\gg{\gamma}
  \def\nn{\nabla} \def\pp{\partial} \def\EE{\scr E}
\def\d{\text{\rm{d}}} \def\bb{\beta} \def\aa{\alpha} \def\D{\scr D}
  \def\si{\sigma} \def\ess{\text{\rm{ess}}}
\def\beg{\begin} \def\beq{\begin{equation}}  \def\F{\scr F}
\def\Ric{\text{\rm{Ric}}} \def\Hess{\text{\rm{Hess}}}
\def\e{\text{\rm{e}}} \def\ua{\underline a} \def\OO{\Omega}  \def\oo{\omega}
 \def\tt{\tilde} \def\Ric{\text{\rm{Ric}}}
\def\cut{\text{\rm{cut}}} \def\P{\mathbb P} \def\ifn{I_n(f^{\bigotimes n})}
\def\C{\scr C}      \def\aaa{\mathbf{r}}     \def\r{r}
\def\gap{\text{\rm{gap}}} \def\prr{\pi_{{\bf m},\varrho}}  \def\r{\mathbf r}
\def\Z{\mathbb Z} \def\vrr{\varrho} \def\ll{\lambda}
\def\L{\scr L}\def\Tt{\tt} \def\TT{\tt}\def\II{\mathbb I}
\def\i{{\rm in}}\def\Sect{{\rm Sect}}\def\E{\mathbb E} \def\H{\mathbb H}
\def\M{\scr M}\def\Q{\mathbb Q} \def\texto{\text{o}} \def\LL{\Lambda}
\def\Rank{{\rm Rank}} \def\B{\scr B} \def\i{{\rm i}} \def\HR{\hat{\R}^d}
\def\to{\rightarrow}\def\l{\ell}
\def\8{\infty}\def\X{\mathbb{X}}\def\3{\triangle}
\def\V{\mathbb{V}}\def\M{\mathbb{M}}\def\W{\mathbb{W}}\def\Y{\mathbb{Y}}\def\1{\lesssim}

\def\va{\varphi}
\def\l{\lambda}
\def\var{\varphi}
\renewcommand{\bar}{\overline}
\renewcommand{\hat}{\widehat}
\renewcommand{\tilde}{\widetilde}
  
\maketitle
\begin{abstract}
Under  a Lipschitz condition on distribution dependent  coefficients, the central limit theorem and the moderate deviation principle are obtained for solutions of McKean-Vlasov type stochastic differential equations, which extend from the corresponding results for classical stochastic differential equations to the distribution dependent setting.
\end{abstract}

\section{Introduction}
In recent years, McKean-Vlasov stochastic differential equations (SDEs for short) have received increasing attentions by researchers. They are also called as mean-field SDEs or distribution dependent SDEs, which are much more involved than classical SDEs as the drift and diffusion coefficients depending on the solution and the law of solution. 
In a nutshell, this kind of equations play important roles in characterising non-linear Fokker-Planck equations and environment dependent financial systems, see \cite{D1,D2,F,G,M,S,V,Y} and references therein. Also, this kind of SDEs have been applied to characterise partial differential equations (PDEs for short) involving the Lions derivative (L-derivative for short), which was introduced by P.-L. Lions in his lecture notes \cite{CP}, see also \cite{B,H,L,R1,R2} for more details. Additionally, \cite{W} investigated the distribution dependent SDEs for Landau type equations. The analysis of stochastic particle systems (that is why McKean-Vlasov equations can be treated as the limiting behaviour of individual particles)  has developed as crucial mathematic tools modelling economic and finance systems.  

It is well known that the key point of LDPs is to show the probability property of a rare event, see \cite{BaoJ,C,DW,HI,Z}. In the case of stochastic process, the idea is to find a deterministic path around which the diffusion is concentrated with high probability. In a nutshell, the stochastic motion can be interpreted as a small perturbation of the deterministic path. There are two main approaches to investigate LDPs, one is weak convergence method, the other one is based on exponential approximation argument. For instance, \cite{DW} investigated the Freidlin-Wentzell LDP in path space for McKean-Vlasov equations and the functional iterated logarithm  law by using techniques of exponential approximation arguments.
 In this paper, we investigate the central limit theorem (CLT for short) and the moderate deviation principle (MDP for short) for  solutions of  distribution dependent SDEs by using the weak convergence approach. It is worth noting that the weak convergence approach results in a convenient representation formula for the LDPs rate function.

The motivation of the MDP study comes from \cite{R2}, which investigate the Bismut formula for Lions derivative of distribution dependent SDEs and applications under the Lipschitz conditions on coefficients.

Let $\mathscr{P}(\R^d)$ be the space of all probability measures on $\R^d$,  consider the following distribution dependent SDE on $\R^d$:
\begin{equation}\label{eq1.2}
\d X_t^\ep=b_t(X_t^\ep,\mathscr{L}_{X_t^\ep})\d t
+\ss\ep\sigma_t(X_t^\ep,\mathscr{L}_{X_t^\ep})\d W_t,
~~X_0^\ep=x,
\end{equation}
where $W_t$ is the $d$-dimensional Brownian motion defined on a complete filtered probability space $(\Omega,\{\mathscr{F}_t\}_{t\ge 0},\P)$, $\mathscr{L}_{X_t^\ep}$ is the law of  $X_t^\ep$,
and $$b:[0,\8)\times\R^d\times\mathscr{P}(\R^d)\rightarrow\R^d,~~\sigma:[0,\8)\times\R^d\times\mathscr{P}(\R^d)\rightarrow\R^{d\otimes d}.$$

To give the main results, in the sequel, we first recall the theory of LDP.

Consider the Cameron-Martin space associated with the Brownian motion $\{W_t; t\in [0,T]\}$, the space of all absolutely continuous paths on the interval $[0,T]$ which starts at $0$ and have derivative almost everywhere which is $L^2([0,T])$ integrable, that is,
$$\H=\Big\{h\in C ([0,T];\R^d):h(0)={\bf 0}, h(\cdot)=\int_0^\cdot\dot{h}(s)\d s; \dot{h}\in L^2([0,T];\R^d)\Big\}.$$
 It is again a Hilbert space with inner product $\langle h_1,h_2\rangle_{\H}:=\int_0^T\langle \dot{h}_1(s),\dot{h}_2(s)\rangle\d s$. 
 Let $\mathscr{A}$ denote the class of $\R^d$ valued $\{\mathscr{F}_t\}$-predictable processes $h(\omega,\cdot)$ belonging to $\H$ a.s. Let 
 $$S_N:=\{h\in\H;~\int_0^T|\dot{h}(s)|^2\d s\le N\}.$$
 $S_N$ is endowed with the weak topology induced from $\H$. Define
 $$\mathscr{A}_N:=\{h\in\mathscr{A},h(\omega)\in S_N,~\P-a.s.\}.$$
 
 We also recall the definition of $L$-derivative. Let 
$$\mathscr{P}_2(\R^d)=\Big\{\mu\in\mathscr{P}(\R^d)
:\mu(|\cdot|^2):=\int_{\R^d}|x|^2\mu(\d x)<\8\Big\}.$$
Then $\mathscr{P}_2(\R^d)$ is a Polish space under the Wasserstein distance
$$\mathbb{W}_2(\mu,\nu):=\inf_{\pi\in\mathscr{C}(\mu,\nu)}
\Big(\int_{\R^d\times\R^d}|x-y|^2\pi(\d x,\d y)\Big)^{\frac{1}{2}},~~\mu,\nu\in\mathscr{P}_2(\R^d),$$
where $\mathscr{C}(\mu,\nu)$ is the set of couplings of $\mu,\nu$; that is, a probability measure $\pi$ on the product space $(\R^d\times\R^d, \mathscr{F}\times\mathscr{F})$ such that $\pi(\cdot\times\R^d)=\mu$ and $\pi(\R^d\times\cdot)=\nu$.
Moreover, the Wasserstein metric induces a topology on $\mathscr{P}_2(\R^d)$, which has been shown to be the topology of weak convergence of measure together with the convergence of all moments of order up to $2$, see \cite[Chapter5]{CR}. We will use $\bf{0}$ denote vectors with components $0$.

 \begin{defn}\label{def1}
 Let $f:\mathscr{P}_2(\R^d)\rightarrow\R$.
 \begin{enumerate}
 \item[(1)] $f$ is called $L$-differentiable at $\mu\in\mathscr{P}_2(\R^d)$, if the functional
 $$L^2(\R^d\rightarrow\R^d,\mu)\ni\phi\mapsto f(\mu\circ(Id+\phi)^{-1})$$
 is Fr\'echet differentiable at ${\bf 0}\in L^2(\R^d\rightarrow\R^d,\mu)$; that is, there exists a
 $\gamma\in L^2(\R^d\rightarrow\R^d,\mu)$ such that
\begin{equation}\label{eq1.1}
\lim_{\mu(|\phi|^2)\rightarrow 0}\frac{|(f)(\mu\circ(Id+\phi)^{-1})-(f)(\mu)-\mu(\langle\phi,\gamma\rangle)|}
{\ss{\mu(|\phi|^2)}}=0,
\end{equation} 
where $\mu(\langle\phi,\gamma\rangle):=\int_{\R^d}\langle\phi(\xi),\gamma(\xi)\rangle\mu(\d\xi)$.
In this case, we donote $D^Lf(\mu)=\gamma$ and call it the $L$-derivative of $f$ at $\mu$.

\item[(2)] If the $L$-derivative $D^Lf(\mu)$ exists for all $\mu\in\mathscr{P}_2(\R^d)$, then $f$ is called $L$-differentiable. If moreover, for every $\mu\in\mathscr{P}_2(\R^d)$ there exists a $\mu$-version $D^Lf(\mu)(\cdot)$ such that $D^Lf(\mu)(x)$ is jointly continuous in $(x,\mu)\in\R^d\times\mathscr{P}_2(\R^d)$, we denote $f\in C^{(1,0)}
(\mathscr{P}_2(\R^d))$.
 \end{enumerate}
 \end{defn}

In this paper, we use the symbol $"\Rightarrow"$ to denote convergence in distribution.

 The following uniform LDP criteria was presented in \cite{LW}.
 \begin{lem}\label{thm1}
 For any $\ep>0$, let $\Gamma^\ep$ be a measurable mapping from $C([0,T];\R^d)$ into $C([0,T];\R^d)$. Suppose that $\{\Gamma^\ep\}_{\ep>0}$ satisfies the following assumptions: 
 There exists a measurable map $\Gamma^0:C([0,T];\R^d)\rightarrow C([0,T];\R^d)$ such that 
 \begin{enumerate}
 \item[\rm{(a)}] For every $N<+\8$ and any family $\{h_\ep;\ep>0\}\subset\mathscr{A}_N$ satisfying that $h_\ep$ converges in distribution as $S_N$-valued random variables to $h$ as $\ep\rightarrow0$, then $$\Gamma^\ep\Big(W_\cdot+\frac{1}{\ss\ep}\int_0^\cdot\dot{h}_\ep(s)\d s\Big)\Rightarrow\Gamma^0\Big(\int_0^\cdot\dot{h}(s)\d s\Big)  \mbox{ as }\ep\rightarrow0.$$
 \item[\rm{(b)}] For every $N<+\8$, the set $\{\Gamma^0(\int_0^\cdot\dot{h}(s)\d s); h\in S_N\}$ is a compact subset of  $C([0,T];\R^d)$.
 \end{enumerate}
 Then the family $\{\Gamma^\ep(W_\cdot)\}_{\ep>0}$ satisfies a large deviation principle in $C([0,T];\R^d)$ with the rate function $I$ given by 
 \begin{equation}\label{eq2.1}
 I(g):=\inf_{h\in\H;g=\Gamma^0(\int_0^\cdot\dot{h}(s)\d s)}
 \Big\{\frac{1}{2}\int_0^T|\dot{h}(s)|^2\d s\Big\},~~g\in C([0,T];\R^d)
 \end{equation}
 with $\inf\emptyset=\8$ by convention.
 \end{lem}

The rest of the paper is organised as follows. In Sect.2 we give the assumptions and the main results Theorem \ref{CLT} and \ref{thm2}; Sect.3 are Sect.4 are devoted to the proofs of Theorem \ref{CLT} and \ref{thm2}, respectively.

Throughout this paper,  we let $C(\alpha, \beta)$ stand for a general constant which depends on parameters $\alpha, \beta$, and may change from occurrence to occurrence.
\section{Main results}
We make the following assumptions about \eqref{eq1.2}.
\begin{enumerate}
\item[({\bf H1})]
The coefficients $b:[0,\8)\times\R^d\times\mathscr{P}_2(\R^d)\rightarrow\R^d,~~\sigma:[0,\8)\times\R^d\times\mathscr{P}_2(\R^d)\rightarrow\R^{d\otimes d}$ are continuous. 
There exists a increasing function $K:[0,\8)\rightarrow[0,\8)$ such that
\begin{equation}\label{eq1.3}
\begin{split}
&|b_t(x,\mu)-b_t(y,\nu)|+\|\sigma_t(x,\mu)-\sigma_t(y,\nu)\|\\
&\le K(t)(|x-y|+\W_2(\mu,\nu)),~~t\ge0,x,y\in\R^d,\mu,\nu\in\mathscr{P}_2(\R^d),
\end{split}
\end{equation}
and 
\begin{equation}\label{eq1.4}
|b_t({\bf 0},\delta_{{\bf 0}})|+\|\sigma_t({\bf 0},\delta_{{\bf 0}})\|
\le K(t),~~t\ge0,
\end{equation}
where and in what follows, for $x\in\R^d$, $\delta_x$ stands for the Dirac measure at $x$, $\|\cdot\|$ is the operator norm.
  \item[({\bf H2})] The coefficient $b_t(x,\mu)$ are differentiable with respect to $x$ and $\mu$ respectively, and its derivative functions satisfy 
  \begin{equation*}
  \begin{split}
  &|\nabla b_t(\cdot,\mu)(x)-\nabla b_t(\cdot,\nu)(y)|\le K(t)(|x-y|+\W_2(\mu,\nu)),\\
  &\|D^Lb_t(x,\cdot)(\mu)-D^Lb_t(y,\cdot)(\nu)\|\le K(t)(|x-y|+\W_2(\mu,\nu)),\\
  &\max\{\|\nabla b_t(\cdot,\mu)(x)\|,\|D^Lb_t(x,\cdot)(\mu)\|\}\le K(t)
  \end{split}
  \end{equation*}
  hold for all $t\ge0, (x,\mu)\in\R^d\times\mathscr{P}_2(\R^d)$.
  \end{enumerate}
\begin{rem}
In our setting, $b_t(x,\mu)=(b_t^i(x,\mu))_{i=1,\cdots,d}$ is a $\R^d$-valued function, and $\sigma_t(x,\mu)=(\sigma_t^{ij}(x,\mu))_{i,j=1,\cdots,d}$ is a $\R^{d\otimes d}$-valued function, thus, we write $D^L b_t(x,\mu)=(D^Lb_t^i(x,\mu))_{i=1,\cdots,d}$.
\end{rem}

Intuitively, as the parameter $\ep$ tends to $0$ in \eqref{eq1.2}, the diffusion term vanishes and we have the following ordinary differential equation
\begin{equation}\label{eq1.6}
\d X_t^0=b_t(X_t^0,\delta_{X_t^0})\d t,
\end{equation}
with the same initial datum as \eqref{eq1.2}, that is, $X_0^0=x$.
Since $x$ is deterministic, we deduce that $\delta_{X_\cdot^0}$ is a Dirac measure centered on the path $X_\cdot^0$.

In the following, we shall investigate the deviations of $X^\ep$ from the solution $X^0$ of ordinary differential equation, that is, the asymptotic behaviour of the trajectory,
\begin{align}\label{eq1.7}
\bar{X}_t^\ep=\frac{1}{\ss\ep\lambda(\ep)}(X_t^\ep-X_t^0),~t\in[0,T].
\end{align}
\begin{enumerate}
\item[(LDP)] The case $\lambda(\ep)=1/{\ss\ep}$ provides some large deviation estimates. \cite{DW} proved that the law of the  solution $X^\ep$ satisfies an LDP by means of the discussion of exponential tightness.
\item[(CLT)] If $\lambda(\ep)\equiv 1$, we are in the domain of the CLT. We will show that $\frac{X^\ep-X^0}{\ss\ep}$ converges to a stochastic process as $\ep\rightarrow0$, see Theorem \ref{CLT}.
\item[(MDP)] To fill in the gap between the CLT scale and the LDP scale, we will study the MDP, that is,  the deviation scale $\lambda(\ep)$ satisfies
\begin{align}\label{eq1.8}
\lambda(\ep)\rightarrow\8,~~\ss\ep\lambda(\ep)\rightarrow 0, ~~\mbox{as}~\ep\rightarrow 0.
\end{align}
In the MDP case, we will prove that $\{\bar{X}^\ep;\ep\in(0,1)\}$ satisfies an LDP, see Theorem \ref{thm2} below.
\end{enumerate}

Our first main result is the following central limit theorem.
\begin{thm}\label{CLT}
Under assumptions ({\bf H1}) and ({\bf H2}),
\begin{align*}
\E\Big(\sup_{0\le t\le T}\Big|\frac{X^\ep(t)-X^0(t)}{\ss\ep}-Z(t)\Big|^p\Big)\1\ep,
\end{align*} 
 where $Z$ is determined by 
\begin{align}\label{Y}
\d Z_t=\nabla_{Z_t}b(X_t^0,\delta_{X_t^0})\d t+\E\langle
D^Lb_t(y,\cdot)(\delta_{X_t^0})(X_t^0),Z_t\rangle|_{y=X_t^0}\d t
+\sigma(X_t^0,\delta_{X_t^0})\d W_t.
\end{align}
\end{thm}
Our second result is that $X^\ep, \ep\in(0,1)$ satisfies the MDP, that is the following theorem.
\begin{thm}\label{thm2}
Under assumptions ({\bf H1}) and ({\bf H2}), $\bar{X}_\cdot^\ep=\frac{X_\cdot^\ep-X_\cdot^0}{\ss\ep\lambda(\ep)}$ satisfies an LDP on $C([0,T];\R^d)$ with the speed $\lambda^2(\ep)$ and with the rate function $I$, which is defined as follows:
\begin{equation}\label{rate}
 I(g):=\inf_{\{h\in\H;g=\Gamma^0(\int_0^\cdot\dot{h}(s)\d s)\}}
 \Big\{\frac{1}{2}\int_0^T|\dot{h}(s)|^2\d s\Big\},~~g\in C([0,T];\R^d),
 \end{equation}
where $\Gamma^0(\int_0^\cdot\dot{h}(s)\d s):=Y_\cdot^h$ satisfies the following equation:
 \begin{equation}\label{eq3.1}
\d Y_t^h=\Big\{\nabla_{Y_t^h} b_t(\cdot,\delta_{X_t^0})(X_t^0)+\sigma_t(X_t^0,\delta_{X_t^0})\dot{h}(t)\Big\}\d t.
\end{equation}
\end{thm}

 \section{Proof of Theorem \ref{CLT}} 
  We first recall a formula of $L$-derivative due to \cite{R2}.
  \begin{lem}\label{lem1}
  Let $(\Omega,\mathscr{F},\P)$ be an atomless probability space, and let $X,Y\in L^2(\Omega\rightarrow\R^d,\P)$ with $\mathscr{L}_X=\mu$. If either $X$ and $Y$ are bounded and $f$ is $L$-differentiable at $\mu$, or $f\in C^{(1,0)}(\mathscr{P}_2(\R^d))$, then 
  \begin{equation}\label{d}
  \lim_{\ep\rightarrow0}\frac{f(\mathscr{L}_{X+\ep Y})-f(\mu)}{\ep}=\E\<D^Lf(\mu)(X),Y\>.
  \end{equation} 
  Consequently,
  \begin{equation}\label{n}
  \Big|\lim_{\ep\downarrow0}\frac{f(\mathscr{L}_{X+\ep Y})-f(\mu)}{\ep}\Big|=|\E\<D^Lf(\mu)(X),Y\>|\le\E\|D^Lf(\mu)\|\ss{\E|Y|^2}.
  \end{equation}
  \end{lem}  
  The existence and uniqueness of solution to \eqref{eq1.2} has been proved in \cite{W}. The following Lemma gives the uniformly $p$-th moment estimates of solutions to \eqref{eq1.2} and \eqref{eq1.6}. 
 \begin{lem}\label{lem3.5}
 Under assumption ({\bf H1}), for $X_0^0=X_0^\ep=x\in\R^d$, we have
 \begin{equation}\label{eq3.8}
 \E\Big(\sup_{0\le t\le T}|X_t^\ep|^p\Big)\vee\Big(\sup_{0\le t\le T}|X_t^0|^p\Big)<\8,~~~~~p\ge 2.
 \end{equation}
 \end{lem}
 \begin{proof}
 It is easy to get from ({\bf H1}),
 \begin{equation}\label{eq3.9}
 |b_t(x,\mu)|\vee\|\sigma_t(x,\mu)\|\le K(t)(1+|x|+\W_2(\mu,\delta_0)).
 \end{equation}
 
 Noting that $\W_2(\mathscr{L}_{X_s^\ep},\delta_0)^p\le(\E|X_s^\ep|^2)^{p/2}$, by the Burkholder-Davis-Gundy (BDG for short) inequality and \eqref{eq3.9}, one has
 \begin{equation*}
 \begin{split}
 \E\Big(\sup_{0\le t\le T}|X_t^\ep|^p\Big)
 &\le 3^{p-1}x^{p}+C(T,p)\E\int_0^T(1+|X_s^\ep|^p)\d s,
 \end{split}
 \end{equation*} 
 thus, the desired assertion follows from Gronwall's inequality.
 \end{proof} 
{\bf Proof of Theorem \ref{CLT}}
 \begin{proof}
 For notation brevity, we set $Z_\cdot^\ep:=\frac{X_\cdot^\ep-X_\cdot^0}{\ss\ep}$, then 
 \begin{align}\label{Z}
 \d Z_t^\ep=\frac{1}{\ss\ep}(b_t(X_t^\ep,\mathscr{L}_{X_t^\ep})-b_t(X_t^0,\delta_{X_t^0}))\d t+\sigma_t(X_t^\ep,\mathscr{L}_{X_t^\ep})\d W_t.
 \end{align}
 We are going to prove $\lim_{\ep\rightarrow0}\E\Big(\sup_{0\le t\le T}|Z_t^\ep-Z_t|^p\Big)=0$. To this end, we claim that 
 \begin{align}\label{X}
 \E\Big(\sup_{0\le t\le T}|X_t^\ep-X_t^0|^p\Big)\le C(T,p).
 \end{align}
 Notice that
\begin{align*}
&\E\Big(\sup_{0\le t\le T}|X_t^\ep-X_t^0|^p\Big)\\
&\le2^{p-1}\Big\{\E\Big|\int_0^T(b_s(X_s^\ep,\mathscr{L}_{X_s^\ep})-b_s(X_s^0,\delta_{X_s^0}))\d s\Big|^p+\ep^{p/2}\E\Big(\sup_{0\le t\le T}\Big|\int_0^t\sigma_s(X_s^\ep,\mathscr{L}_{X_s^\ep})\d W_s\Big|^p\Big\}\\
&=:2^{p-1}(I_1(T)+I_2(T)).
\end{align*}
With Lemma \ref{lem3.5} in hand, we know the boundedness of  $p$-th moment of $X_t^\ep$ and $X_t^0, t\in[0,T]$ and the coefficient $b$ is $L$-differentiable at $\mathscr{L}_{X_\cdot^0}$.  By Lemma \ref{lem1},  \eqref{eq1.4} and assumption ({\bf H2}), we have
 \begin{align*}
 I_1(T)&\le(2T)^{p-1}\Big\{\int_0^T\E|(b_s(X_s^\ep,\mathscr{L}_{X_s^\ep})-b_s(X_s^0,\mathscr{L}_{X_s^\ep}))|^p\d s\\
&~~ +\int_0^T\E|(b_s(X_s^0,\mathscr{L}_{X_s^\ep})-b_s(X_s^0,\delta_{X_s^0}))|^p\d s
\Big\}\\
&\le(2T)^{p-1}\Big\{\int_0^T\E\Big|\int_0^1\nabla b_s(R_s^\ep(r),\mathscr{L}_{X_s^\ep})|X_s^\ep-X_s^0|\d r\Big|^p\d s\\
&~~+\E\Big|\int_0^1\langle D^Lb_s(y,\cdot)(\mathscr{L}_{R_s^\ep(r)})(R_s^\ep(r)),X_s^\ep-X_s^0\rangle|_{y=X_s^0}\d r\Big|^p\d s\Big\} \\
&\le C(T,p)\int_0^T\E|X_s^\ep-X_s^0|^p\d s,
\end{align*}
where $R_s^\ep(r)=X_s^0+r(X_s^\ep-X_s^0), r\in[0,1]$.

By the assumption ${(\bf H1)}$ and BDG's inequality, we get
\begin{align*}
 I_2(T)
 &\le \ep^{p/2}C(T,p)\Big(\int_0^T\big(\E|\sigma_s(X_s^\ep,\mathscr{L}_{X_s^\ep})-\sigma_s(0,\delta_0)|^2+\E|\sigma_s(0,\delta_0)|^2\big)\d s\Big)^{p/2}\\
 &\le \ep^{p/2}C(T,p)\Big(\int_0^T(\E|X_s^\ep|^2+K^2(s))\d s\Big)^{p/2}\\
 &\le\ep^{p/2}C(T,p)(1+\int_0^T\E|X_s^\ep|^p\d s),
 \end{align*}
 where the third inequality due to the fact that $\W_2(\mathscr{L}_{X_s^\ep},\delta_0)^2\le \E|X_s^\ep|^2$.
 
The claim  follows by combining the above the estimates,  \eqref{eq3.8}  and the Gronwall inequality.

 By the definitions of $Z_t^\ep$ and  $Z_t$, we derive that 
 \begin{align*}
 Z_t^{\ep}-Z_t&=\int_0^t\Big(\frac{1}{\ss\ep}(b_s(X_s^\ep,\mathscr{L}_{X_s^\ep})-b_s(X_s^0,\mathscr{L}_{X_s^\ep}))-
 \nabla_{Z_s^\ep}b_s(X_s^0,\mathscr{L}_{X_s^\ep})
 \Big)\d s\\
 &~~+\int_0^t\Big(\frac{1}{\ss\ep}(b_s(X_s^0,\mathscr{L}_{X_s^\ep})-b_s(X_s^0,\delta_{X_s^0}))-\E
 \langle D^Lb_s(y,\cdot)(\delta_{X_s^0})(X_s^0), Z_s^\ep\rangle|_{y=X_s^0} \Big)\d s\\
&~~+\int_0^t(\sigma_s(X_s^\ep,\mathscr{L}_{X_s^\ep})-\sigma_s(X_s^0,\delta_{X_s^0}))\d W_s\\
&~~+\int_0^t(\nabla_{Z_s^\ep}b_s(X_s^0,\mathscr{L}_{X_s^\ep})-\nabla_{Z_s}b_s(X_s^0,\delta_{X_s^0}))\d s\\
&~~+\int_0^t(\E\langle D^Lb_s(y,\cdot)(\delta_{X_s^0})(X_s^0), Z_s^\ep\rangle|_{y=X_s^0}-\E\langle D^Lb_s(y,\cdot)(\delta_{X_s^0})(X_s^0), Z_s\rangle|_{y=X_s^0}
)\d s.
\end{align*}
By ({\bf H2}), \eqref{d}, H\"older's inequality and BDG's inequality, we have
 \begin{align}\label{appr}
 &\E\Big(\sup_{0\le t\le T}|Z_t^{\ep}-Z_t|^p\Big)\\\nonumber
 &\le C(T,p)\int_0^T\E
 \Big|\int_0^1\nabla_{Z_s^\ep}b_s(R_s^\ep(r),\mathscr{L}_{X_s^\ep})\d r-\nabla_{Z_s^\ep}b_s(X_s^0,\mathscr{L}_{X_s^\ep})\Big|^p\d s\\\nonumber
 &~~+C(T,p)
 \int_0^T\E\Big|\int_0^1\langle D^Lb_s(y,\cdot)(\mathscr{L}_{R_s^\ep(r)})(R_s^\ep(r)), Z_s^\ep\rangle|_{y=X_s^0}\d r\\\nonumber
 &-\langle D^Lb_s(y,\cdot)(\mathscr{L}_{X_s^0})(X_s^0), Z_s^\ep\rangle|_{y=X_s^0}\Big|^p\d s
+C(T,p)\int_0^T\E|\sigma_s(X_s^\ep,\mathscr{L}_{X_s^\ep})-\sigma_s(X_s^0,\mathscr{L}_{X_s^0})|^p\d s\\\nonumber
&~~+C(T,p)\int_0^T\Big(\E|\nabla_{Z_s^\ep-Z_s}b_s(X_s^0,\mathscr{L}_{X_s^\ep})|^p+\E|\nabla_{Z_s}b_s(X_s^0,\mathscr{L}_{X_s^\ep})-\nabla_{Z_s}b_s(X_s^0,\delta_{X_s^0})|^p\Big)\d s\\\nonumber
&~~+C(T,p)\int_0^T\Big|\E\langle D^Lb_s(y,\cdot)(\delta_{X_s^0})(X_s^0), Z_s^\ep\rangle|_{y=X_s^0}-\E\langle D^Lb_s(y,\cdot)(\delta_{X_s^0})(X_s^0), Z_s\rangle|_{y=X_s^0}\Big|^p\d s\\ \nonumber
 &\le C(T,p)\ep^{p/2}\int_0^T\E|Z_s^{\ep}|^{2p}\d s
+C(T,p)\int_0^T\E\Big(|Z_s^\ep|\int_0^1\W_2(\mathscr{L}_{R_s^\ep(r)},\delta_{X_s^0})\d r\Big)^p\d s\\\nonumber
&~~+C(T,p)\int_0^T(\ep^{p/2}\E|Z_s^\ep|^p+\E\W_2(\mathscr{L}_{X_s^\ep},\delta_{X_s^0})^p)\d s\\\nonumber
&~~+C(T,p)\int_0^T\Big(\E|Z_s^\ep-Z_s|^p+\E|Z_s|^p\W_2(\mathscr{L}_{X_s^\ep},\delta_{X_s^0})^p\Big)\d s\\\nonumber
&\le C(T,p)\ep^{p/2}+C(T,p)\int_0^T\E|Z_s^\ep-Z_s|^p\d s,
 \end{align}
where $R_s^\ep(r)=X_s^0+r(X_s^\ep-X_s^0), r\in[0,1]$, and we used $\int_0^1\W_2(\mathscr{L}_{R_s^\ep(r)},\delta_{X_s^0})\d r\le\frac{\ss\ep}{2}(\E|Z_s^\ep|^2)^{1/2}$,  $\W_2(\mathscr{L}_{X_s^\ep},\delta_{X_s^0})\le\ep^{1/2}(\E |Z_s^\ep|^2)^{1/2}$ and \eqref{X} in the last inequality.
  
By the Gronwall inequality, it yields from \eqref{appr} that
\begin{align*}
\E\Big(\sup_{0\le t\le T}|Z_t^{\ep}-Z_t|^p\Big)\le C_{T,p}\ep^{p/2}.
\end{align*}
The desired assertion is obtained by taking $\ep\rightarrow0$.
 \end{proof}
 
\section{Proof of Theorem \ref{thm2}} 

From \eqref{eq1.2}, \eqref{eq1.6}, \eqref{eq1.7}, we can see that $\bar{X}^\ep$ satisfies the following equation:
 \begin{equation}\label{eq3.3}
 \bar{X}_t^\ep=\frac{1}{\ss\ep\lambda(\ep)}\int_0^t[b_s(X_s^\ep,\mathscr{L}_{X_s^\ep})-b_s(X_s^0,\delta_{X_s^0}) ]\d s
 +\frac{1}{\lambda(\ep)}\int_0^t\sigma_s(X_s^\ep,\mathscr{L}_{X_s^\ep})\d W_s.
 \end{equation}
 Notice that the law of $X_t^0$ can be approximated by $\mathscr{L}_{X_t^\ep}$ as $\ep\rightarrow0$, we then define the following equation:
  \begin{equation}\label{app}
 \bar{Y}_t^\ep=\frac{1}{\ss\ep\lambda(\ep)}\int_0^t[b_s(Y_s^\ep,\delta_{X_s^0})-b_s(X_s^0,\delta_{X_s^0}) ]\d s
 +\frac{1}{\lambda(\ep)}\int_0^t\sigma_s(Y_s^\ep,\delta_{X_s^0})\d W_s,
 \end{equation}
 where $\d Y_t^\ep=b_t(Y_t^\ep,\delta_{X_t^0})\d t+\ss\ep\sigma_t(Y_t^\ep,\delta_{X_t^0})\d W_t$ and $\bar{Y}_t^\ep=\frac{Y_t^\ep-X_t^0}{\ss\ep\lambda(\ep)}$.

 We recall that the heuristics underlying large deviations theory is to identify a deterministic path around which the diffusion is concentrated with overwhelming probability, so that the stochastic motion can be seen as a small random perturbation of this deterministic path. This means in particular that the law of $\bar{X}_t^\ep$ is close to some Dirac mass if $\ep$ is small. We therefore proceed in two steps toward the aim of proving a large deviation principle for $\bar{X}^\ep$.
 In the first step, we replace the law $\mathscr{L}_{\bar{X}_t^\ep}$ of $\bar{X}_t^\ep$ by its suspected limit, that is, we approximate the law of $\bar{X}_t^\ep$. In this way, we avoid the difficulty of the dependence on the law of $\bar{X}^\ep$, thus, we obtain a diffusion which is defined by  means of a classical SDE. In the second step, we prove that this diffusion is exponentially equivalent to $\bar{X}^\ep$. Since LDPs does not distinguish between exponentially equivalent families, we deduce that $\bar{X}^\ep$ satisfies an LDP with the good rate function $I(g)$ given in \eqref{rate}.
 
 To make the content self-contained. In the following subsection, we give the sketch proof of LDP for $\bar{Y}^\ep$.
\subsection{Large deviation principle for $\bar{Y}^\ep$}
 \begin{lem}\label{app}
 Under the hypotheses of Theorem \ref{thm2}, the family of $(\bar{Y}^\ep)_{\ep>0}$ satisfies a large deviation principle in $C([0,T];\R^d)$ equipped with the topology of the uniform norm with the good rate function $I(g)$ given in \eqref{rate}.
 \end{lem}
According to the Lemma \ref{thm1},  in order to prove Lemma \ref{app}, we only need to verify the conditions (a) and  (b) in Lemma \ref{thm1}. 

 By the Yamada-Watanabe theorem, there exists a measurable map $\Gamma^\ep:C([0,T];\R^d)\rightarrow C([0,T];\R^d)$ such that 
 $\bar{Y}_\cdot^\ep=\Gamma^\ep\Big(\frac{1}{\lambda(\ep)}W_\cdot\Big)$.
 
Since $\E_\P\Big(\exp\big\{\frac{1}{2}\int_0^T|\dot{h}_\ep(s))|^2\d s\big\}\Big)<\8, h_\ep\in\mathscr{A}_N$, that is, if $h_\ep\in\mathscr{A}_N,$ then the Novikov's condition holds. By the Girsanov theorem, 
we know that
\begin{align*}
\frac{1}{\lambda(\ep)}\tilde{W}_t=\frac{1}{\lambda(\ep)}W_t+\int_0^t\dot{h}_\ep(s)\d s
\end{align*}
is a Brownian motion under the probability measure $\P_\ep:=R_T\P$, where 
\begin{align*}
R_T=\exp{\Big\{-\int_0^T\dot{h}_\ep(s)\d \frac{W_s}{\lambda(\ep)}-\frac{1}{2}\int_0^T|\dot{h}_\ep(s)|^2\d s\Big\}}
\end{align*}
is a martingale.

Furthermore, we obtain that
$\bar{Y}_\cdot^{\ep,h_\ep}=\Gamma^\ep\Big(\frac{1}{\lambda(\ep)}W_\cdot+\int_0^\cdot \dot{h}_\ep(s)\d s\Big)$, 
 which solves
\begin{align}\label{eq3.5}
 \d \bar{Y}_t^{\ep,h_\ep}&=\frac{1}{\ss\ep\lambda(\ep)}[b_t(Y_t^{\ep,h_\ep},\delta_{X_t^{0}})-b_t(X_t^0,\delta_{X_t^{0}}) ]\d t\\\nonumber
 &~~+\frac{1}{\lambda(\ep)}\sigma_t(Y_t^{\ep,h_\ep},\delta_{X_t^{0}})\d W_t
 +\sigma_t(Y_t^{\ep,h_\ep},\delta_{X_t^{0}})\dot{h}_\ep(t)\d t,
\end{align}
where $Y_t^{\ep,h_\ep}:=X_t^0+\ss\ep\lambda(\ep)\bar{Y}_t^{\ep,h_\ep}$.

The following Lemmas play the key roles in the proof of Lemma \ref{app}.
 
 \begin{lem}\label{lem3.1}
 Under Assumptions {\rm({\bf H1})} and {\rm ({\bf H2})}, for any $h\in \H$, equation \eqref{eq3.1} admits a unique solution $Y_\cdot^h$ in $C([0,T];\R^d)$.
 Moreover, for any $N>0$, there exists a constant $C_{N,T}$ such that 
 \begin{equation}\label{eq3.2}
 \sup_{h\in S_N}\Big\{\sup_{0\le t\le T}|Y_t^h|\Big\}\le C_{N,T}.
 \end{equation}
 \end{lem}
 \begin{proof}
 By ({\bf H1}) and  ({\bf H2}),  the coefficients of \eqref{eq3.1} satisfy the Lipschitz condition,  therefore equation \eqref{eq3.1} admits a unique solution. Moreover, noting   the coefficient functions satisfy the linear growth condition and  the fact that $\W_2(\mathscr{L}_{Y_t^h},\delta_0)^2\le|Y_t^h|^2$, we can obtain the estimate \eqref{eq3.2} by using the Gronwall inequality. Here we omit the details of the proof.
 \end{proof}
 Firstly, we prove that the condition (b) of Lemma \ref{thm1} holds.
 \begin{lem}\label{lem3.2}
 Under assumptions {\rm({\bf H1})} and {\rm ({\bf H2})}, for any positive number $N<\8$, 
 the family 
 $$K_N:=\Big\{\Gamma^0\Big(\int_0^\cdot\dot{h}(s)\d s\Big); h\in S_N\Big\},$$ 
 is compact in $C([0,T];\R^d)$, where the map $\Gamma^0$ is defined in Theorem \ref{thm2}.
 \end{lem}
 \begin{proof}
If the map $\Gamma^0$  is continuous from $S_N$ to $C([0,T];\R^d)$. Then for any $N<\8$, the fact that $K_N$ is compact follows from the compactness of $S_N$ and the continuity of the map $\Gamma^0$ from $S_N$ to $C([0,T];\R^d)$.
 
 In the sequel, we prove that $\Gamma^0$ is a continuous map from
 $S_N$ to $C([0,T];\R^d)$. Let $h_n\rightarrow h$  in $S_N$ as $n\rightarrow\8$. Then
 \begin{equation*}
 \begin{split}
 Y_t^{h_n}-Y_t^h&=\int_0^t\nabla_{\{Y_s^{h_n}-Y_s^h\}}b_s(\cdot,\delta_{X_s^0})(X_s^0)\d s
 +\int_0^t\sigma_s(X_s^0,\delta_{X_s^0})(\dot{h}_n(s)-\dot{h}(s))\d s\\
 &=:I_1^n(t)+I_2^n(t).
 \end{split}
 \end{equation*}
 By ({\bf H2}), \eqref{eq3.8} and \eqref{eq3.9}, it is easy to see that
 \begin{align*}
 |I_1^n(t)|\le\int_0^tK(s)(1+|X_s^0|+\W_2(\delta_{X_s^0},\delta_0))|Y_s^{h_n}-Y_s^h|\d s.
 \end{align*}
 
Let $g^n(t)=\int_0^t\sigma_s(X_s^0,\delta_{X_s^0})\dot{h}_n(s)\d s$.  By {\bf(H1)},  Lemma \ref{lem3.5}, and $h_n,h\in S_N$, we derive that
  \begin{align*}
 |g^n(t)|&\le\Big(\int_0^t\|\sigma_s(X_s^0,\delta_{X_s^0})\|^2\d s\Big)^{1/2}\Big(\int_0^t|\dot{h}_n(s)|^2\d s\Big)^{1/2}\\
&\le\Big(\int_0^tK^2(s)(1+|X_s^0|+W_2(\delta_{X_s^0},\delta_0))^2\d s\Big)^{1/2}\Big(\int_0^t|\dot{h}_n(s)|^2\d s\Big)^{1/2}\\
 &<\8. 
 \end{align*}
Similarly, we see that for any $0\le t_1\le t_2\le T$,
\begin{align*}
|g^n(t_2)-g^n(t_1)|&\le\int_{t_1}^{t_2} \|\sigma_s(X_s^0,\delta_{X_s^0})\||\dot{h}_n(s)|\d s\\
 &\le\int_{t_1}^{t_2}K(s)(1+|X_s^0|+W_2(\delta_{X_s^0},\delta_0))|\dot{h}_n(s)|\d s\\
 &\le C(T)(t_2-t_1)^{1/2}\Big(\int_{t_1}^{t_2}|\dot{h}_n(s)|^2\d s\Big)^{1/2}\\
 &\le C(T,N)(t_2-t_1)^{1/2}.
\end{align*}
 Hence, the family of function $\{g^n\}_{n\ge1}$ are equicontinuous in $C([0,T];\R^d)$.
  
 According to the Azel$\grave{a}$-Ascoli theorem, $\{g^n\}_{n\ge1}$ is relatively compact in $C([0,T];\R^d)$, let $g$ be any limit point of $\{g^n\}_{n\ge1}$. Noticing ${h}_n\rightarrow{h}$ on $S_N$, we have
 \begin{align*}
 \lim_{n\rightarrow\8}\int_0^t\sigma_s(X_s^0,\delta_{X_s^0})\dot{h}_n(s)\d s=\int_0^t\sigma_s(X_s^0,\delta_{X_s^0})\dot{h}(s)\d s,   \forall t\in[0,T],
 \end{align*}
 that is,
 $\lim_{n\rightarrow\8}\sup_{t\in[0,T]}|I_2^n(t)|=0$.
 This, together with \eqref{eq3.8},  yields that
 \begin{align*}
 \sup_{0\le t\le T}|Y_t^{h_n}-Y_t^h|&\le \int_0^TK(t)(1+|X_t^0|+\W_2(\delta_{X_t^0},\delta_0))|Y_t^{h_n}-Y_t^h|\d t+\sup_{0\le t\le T}I_2^n(t),
 \end{align*}
 by the Gronwall inequality, we arrive at 
 \begin{align*}
 \sup_{0\le t\le T}|Y_t^{h_n}-Y_t^h|&\le\exp{\Big\{\int_0^TK(t)(1+|X_t^0|+\W_2(\delta_{X_t^0},\delta_0))\d t\Big\}}\sup_{0\le t\le T}I_2^n(t)\\
 &\le C(T,N)\sup_{0\le t\le T}I_3^n(t)\rightarrow 0, \mbox{as}~~n\rightarrow\8.
 \end{align*}
 Thus we proved that the $\Gamma^0$ is a continuous map, 
 the proof is therefore completed.
 \end{proof}

 Before verify condition (a), we give an  estimate for the second moment  of $\bar{Y}_t^{\ep,h_\ep}$.
 \begin{lem}\label{lem3.3}
 For every fixed $N\in\N$, let $h_\ep\in\mathscr{A}_N$ and $\bar{Y}_\cdot^{\ep,h_\ep}$ be given by \eqref{eq3.5}. Then  we have
\begin{equation}\label{eq3.6}
\E\Big(\sup_{0\le t\le T}|\bar{Y}_t^{\ep,h_\ep}|^2\Big)\le C(N,T),
\end{equation}
where $C(N,T)$ is a constant which is independent of $\ep$. 
 \end{lem}
 \begin{proof}
By \eqref{eq3.5},  we have
\begin{align*}
 \bar{Y}_t^{\ep,h_\ep}&=\int_0^t\frac{1}{\ss\ep\lambda(\ep)}[b_s(Y_s^{\ep,h_\ep},\delta_{X_s^0})-b_s(X_s^0,\delta_{X_s^0}) ]\d s\\
 &~~+\int_0^t\frac{1}{\lambda(\ep)}\sigma_s(Y_s^{\ep,h_\ep},\delta_{X_s^0})\d W_s
 +\int_0^t\sigma_s(Y_s^{\ep,h_\ep},\delta_{X_s^0})\dot{h}_\ep(s)\d s\\
 &=: \sum_{i=1}^3J_i^{\ep,h_\ep}(t).
\end{align*}
By ({\bf H2}), we have that
 \begin{equation*}
 \begin{split}
\E\Big(\sup_{0\le t\le T}|J_1^{\ep,h_\ep}(t)|^2\Big) 
&\le\frac{T}{\ep\lambda^2(\ep)}\int_0^T\E\Big|\int_0^1\nabla b_s(X_s^0+r(Y_s^{\ep,h_\ep}-X_s^0),\delta_{X_s^0})|Y_s^{\ep,h_\ep}-X_s^0|\d r\Big|^2\d s\\
&\le C(N,T)\int_0^T\E|\bar{Y}_s^{\ep,h_\ep}|^2\d s.
\end{split}
\end{equation*}
 By the BDG inequality, \eqref{eq3.8} and \eqref{eq3.9}, one has
 \begin{equation*}
 \begin{split}
\E\Big(\sup_{0\le t\le T}|J_2^{\ep,h_\ep}(t)|^2\Big) 
&\le\frac{4}{\lambda^2(\ep)}\int_0^TK^2(s)\E[1+|Y_s^{\ep,h_\ep}|^2+\W_2(\delta_{X_s^0},\delta_0)^2]\d s\\
&\le\frac{C(N,T)}{\lambda^2(\ep)}\int_0^T[1+\E|Y_s^{\ep,h_\ep}-X_s^0|^2+\E|X_s^0|^2]\d s\\
&\le\frac{C(N,T)}{\lambda^2(\ep)}\int_0^T[1+\ep\lambda^2(\ep)\E|\bar{Y}_s^{\ep,h_\ep}|^2+\E|X_s^0|^2]\d s\\
&\le\frac{C(N,T)}{\lambda^2(\ep)}+\ep C(N,T)\int_0^T\E|\bar{Y}_s^{\ep,h_\ep}|^2\d s.
 \end{split}
 \end{equation*}
 For $J_3^{\ep,h_\ep}(t)$, it follows from  ({\bf H1}), \eqref{eq3.8} and $h_\ep\in\mathscr{A}_N$ that
 \begin{align*}
&\E\Big(\sup_{0\le t\le T}|J_3^{\ep,h_\ep}(t)|^2\Big) \\
&=\E\Big|\int_0^T\Big[\sigma_s(Y_s^{\ep,h_\ep},\delta_{X_s^0})-\sigma_s(X_s^0,\delta_{X_s^0})
+\sigma_s(X_s^0,\delta_{X_s^0})\Big]\dot{h}_\ep(s)\d s\Big|^2\\
&\le C(T)\int_0^TK^2(s)\E[1+|X_s^0|^2+\W_2(\delta_{X_s^0},\delta_0)^2]|\dot{h}_\ep(s)|^2\d s\\
&~~+C(T)\ep\lambda^2(\ep)\int_0^T K^2(s)\E|\bar{Y}_s^{\ep,h_\ep}|^2|\dot{h}_\ep(s)|^2\d s\\
&\le C(N,T)\Big(1+\Big(\sup_{0\le t\le T}|X_t^0|^2\Big)+\ep\lambda^2(\ep)\E\Big(\sup_{0\le t\le T}|\bar{Y}_t^{\ep,h_\ep}|^2\Big)\Big)\int_0^T|\dot{h}_\ep(s)|^2\d s\\
&\le C(N,T)\Big(1+\ep\lambda^2(\ep)\E\Big(\sup_{0\le t\le T}|\bar{Y}_t^{\ep,h_\ep}|^2\Big)\Big) .
 \end{align*}
 Thus, we arrived at 
 \begin{equation*}
 \E\Big(\sup_{0\le t\le T}|\bar{Y}_t^{\ep,h_\ep}|^2\Big)
 \le C(N,T)\Big(1+\frac{1}{\lambda^2(\ep)}+\ep\lambda^2(\ep)\E \Big(\sup_{0\le t\le T}|\bar{Y}_t^{\ep,h_\ep}|^2\Big)
 +(1+\ep)\int_0^T\E|\bar{Y}_t^{\ep,h_\ep}|^2\d t\Big).
 \end{equation*}
 Taking $\ep>0$ sufficiently small such that $C(N,T)\ep\lambda^2(\ep)\le\frac{1}{2}$ leads to
 \begin{equation*}
 \E\Big(\sup_{0\le t\le T}|\bar{Y}_t^{\ep,h_\ep}|^2\Big)
 \le C(N,T)\Big(1+\frac{1}{\lambda^2(\ep)}
 +(1+\ep)\int_0^T\E\Big(\sup_{0\le s\le t}|\bar{Y}_s^{\ep,h_\ep}|^2\Big)\d t\Big).
 \end{equation*}
 The desired assertion follows from Gronwall's inequality and due to the fact that $\frac{1}{\lambda^2(\ep)}\rightarrow0$ as $\ep\rightarrow0$.
 \end{proof}

 We are now in the position to verify the condition (a) of Lemma \ref{thm1}.

\begin{lem}\label{lem3.4}
Assume that ({\bf H1}), ({\bf H2}) hold. For every fixed $N\in \N$, let
$h_\ep, h\in\mathscr{A}_N$ be such that $h_\ep$ converges in distribution to $h$ as $\ep\rightarrow0$.
Then $\Gamma^\ep\Big(\frac{1}{\lambda(\ep)}W_\cdot+\int_0^\cdot\dot{h}_\ep(s)\d s\Big)$ converges in distribution to $\Gamma^0\Big(\int_0^\cdot\dot{h}(s)\d s\Big)$ in $C([0,T];\R^d)$.
\end{lem}
\begin{proof}
By the Skorokhod representation theorem \cite[Theorem 6.7, p70]{BP}, there exists a probability space $(\tilde{\Omega},\tilde{\mathscr{F}},\tilde{\mathscr{F}}_t,\tilde{\P})$, and a Brownian motion $\tilde{W}$ on this basis, a family of 
$\tilde{\mathscr{F}}_t$-predictable processes $\{\tilde{h_\ep};\ep>0\},\tilde{h}$ taking values on $\mathscr{A}_N$, $\tilde{\P}$- a.s., such that the joint law of $(h_\ep,h,W)$ under $\P$ coincides with the law of $(\tilde{h}_\ep,\tilde{h},\tilde{W})$ under $\tilde{\P}$ and 
\begin{align*}
\lim_{\ep\rightarrow0}\langle\tilde{h}_\ep-\tilde{h},g\rangle=0, 
\forall g\in\H,\tilde{\P}- a.s.
\end{align*}
Let $\tilde{Y}^{\ep,\tilde{h}_\ep}$ be the solution of \eqref{eq3.5} replacing $h_\ep$ by $\tilde{h}_\ep$ and $W$ by $\tilde{W}$, and $\tilde{Y}^{\tilde{h}}$ be the solution of \eqref{eq3.1} replacing $h$ by $\tilde{h}$. Thus, to this end, it suffices to verify 
\begin{align*}
\lim_{\ep\rightarrow0}\|\tilde{Y}^{\ep,\tilde{h}_\ep}-\tilde{Y}^{\tilde{h}}\|=0,~~\mbox{in probability}.
\end{align*}
In the sequel, we drop off the $\tilde{\cdot}$ in the notation for the sake of simplicity.

By the definitions of $Y_t^h$, $\bar{Y}_t^{\ep,h_\ep}$, i.e., \eqref{eq3.1}, \eqref{eq3.5}, it yields that
\begin{equation*}
\begin{split}
&\bar{Y}_t^{\ep,h_\ep}-Y_t^h\\
&=\left[\frac{1}{\ss\ep\lambda(\ep)}\int_0^t[b_s(Y_s^{\ep,h_\ep},\delta_{X_s^0})-b_s(X_s^0,\delta_{X_s^0})]\d s
-\int_0^t\nabla_{Y_s^h}b_s(\cdot,\delta_{X_s^0})(X_s^0)\d s\right]\\
&~~+\int_0^t\Big[\sigma_s(Y_s^{\ep,h_\ep},\delta_{X_s^0})\dot{h}_\ep(s)-\sigma_s(X_s^0,\delta_{X_s^0})\dot{h}(s)\Big]\d s
+\frac{1}{\lambda(\ep)}\int_0^t\sigma_s(Y_s^{\ep,h_\ep},\delta_{X_s^0})\d W_s\\
&=:\sum_{i=1}^3I_i^{\ep,h_\ep}(t).
\end{split}
\end{equation*}
By  ({\bf H2}), we have 
\begin{align*}
|I_1^{\ep,h_\ep}(t)|&=\int_0^t\Big|\int_0^1\nabla_{\bar{Y}_s^{\ep,h_\ep}}b_s(\cdot,\delta_{X_s^0})(X_s^0+r(Y_s^{\ep,h_\ep}-X_s^0))\d r-\nabla_{Y_s^h}b_s(\cdot,\delta_{X_s^0})(X_s^0)\Big|\d s\\
&\le\int_0^t\Big|\int_0^1\nabla_{\{\bar{Y}_s^{\ep,h_\ep}-Y_s^h\}}b_s(\cdot,\delta_{X_s^0})(X_s^0+r(Y_s^{\ep,h_\ep}-X_s^0))\d r\Big|\d s\\
&+\int_0^t\Big|\int_0^1\nabla_{Y_s^h}b_s(\cdot,\delta_{X_s^0})(X_s^0+r(Y_s^{\ep,h_\ep}-X_s^0))\d r-\nabla_{Y_s^h}b_s(\cdot,\delta_{X_s^0})(X_s^0)\Big|\d s\\
&\le\int_0^tK(s)|\bar{Y}_s^{\ep,h_\ep}-Y_s^h|\d s
+\int_0^t\frac{\ss\ep\lambda(\ep)}{2}K(s)|Y_s^h||\bar{Y}_s^{\ep,h_\ep}|\d s
\end{align*}
By \eqref{eq3.2}, \eqref{eq3.6}, it follows that
\begin{align*}
\E\Big(\sup_{0\le t\le T}|I_1^{\ep,h_\ep}(t)|^2\Big)
\1\ep\lambda^2(\ep)+\int_0^T\E|Y_s^{\ep,h_\ep}-Y_s^h|^2\d s.
\end{align*}
By ({\bf H2}) and H\"older's inequality, it follows that
\begin{align*}
&|I_2^{\ep,h_\ep}(t)|\\
&\le\Big|\int_0^t\Big[\sigma_s(Y_s^{\ep,h_\ep},\delta_{X_s^0})-\sigma_s(X_s^0,\delta_{X_s^0})\Big]\dot{h}_\ep(s)\d s\Big|+\Big|\int_0^t\sigma_s(X_s^0,\delta_{X_s^0})(\dot{h}_\ep(s)-\dot{h}(s))\d s\Big|\\
&\le \int_0^tK(s)|Y_s^{\ep,h_\ep}-X_s^0||\dot{h}_\ep(s)|\d s
+\int_0^t|\sigma_s(X_s^0,\delta_{X_s^0})(\dot{h}_\ep(s)-\dot{h}(s))|\d s\\
&\le\ss\ep\lambda(\ep)\int_0^tK(s)|\bar{Y}_s^{\ep,h_\ep}||\dot{h}_\ep(s)|\d s
+\int_0^tK(s)(1+|X_s^0|)|\dot{h}_\ep(s)-\dot{h}(s)|\d s,
\end{align*}
thus,
\begin{align*}
\E\Big(\sup_{0\le t\le T}|I_2^{\ep,h_\ep}(t)|^2\Big)
\1\ep\lambda^2(\ep)+\int_0^T\E|\dot{h}_\ep(s)-\dot{h}(s)|^2\d s.
\end{align*}
By the BDG inequality, \eqref{eq3.9} and \eqref{eq3.2}, we arrive at
\begin{align*}
&\E\Big(\sup_{0\le t\le T}|I_3^{\ep,h_\ep}(t)|^2\Big)\\
&\le\frac{1}{\lambda^2(\ep)}\int_0^T\E\Big(\|\sigma_s(Y_s^{\ep,h_\ep},\delta_{X_s^0})-\sigma_s(X_s^{0},\delta_{X_s^0})\|^2+\|\sigma_s(X_s^{0},\delta_{X_s^0})\|^2\Big)\d s\\
&\1\frac{1}{\lambda^2(\ep)}+\ep\int_0^T\E|\bar{Y}_s^{\ep,h_\ep}|^2\d s.
\end{align*}
Taking the above estimates into consideration, it follows that
\begin{align*}
&\E\Big(\sup_{0\le t\le T}|\bar{Y}_t^{\ep,h_\ep}-Y_t^h|^2\Big)\\
&\1\frac{1}{\lambda^2(\ep)}+\ep(\lambda^2(\ep)+1)+\int_0^T\E|\dot{h}_\ep(s)-\dot{h}(s)|^2\d s+\int_0^T\E|\bar{Y}_s^{\ep,h_\ep}-Y_s^h|^2\d s,
\end{align*}
thus, the desired assertion follows from the Gronwall inequality and taking $\ep\rightarrow0$.
\end{proof}
{\bf Proof of Lemma \ref{app}}

The conclusion of Lemma \ref{app} follows from Lemma \ref{thm1}, Lemmas \ref{lem3.2} and \ref{lem3.4}.

\subsection{$\bar{X}^\ep$ and $\bar{Y}^\ep$ are exponentially equivalent}

\begin{lem}\label{equivalence}
For any $\delta>0$, we have
\begin{align}\label{equi}
\limsup_{\ep\rightarrow0}\ep \log\Big(\P\Big\{\sup_{0\le t\le T}|\bar{X}_t^\ep-\bar{Y}_t^\ep|\ge\delta\Big\}\Big)=-\8.
\end{align}
\end{lem}
The proofs of Lemma \ref{equivalence} is based on the following lemma, which corresponds to \cite[Lemma 5.6.18]{DZ}.
\begin{lem}\label{b}
Let $b_t,\sigma_t$ be progressively measurable processes, and let 
\begin{align*}
\d z_t=b_t\d t+\ss\ep\sigma_t\d w_t, \, t\ge 0,
\end{align*}
where $z_0$ is deterministic. Let $\tau_1\in[0,1]$ be a stopping time with respect to the filtration of $\{w_t,t\in[0,1]\}$. Suppose that the coefficients of the diffusion matrix $\sigma$ are uniformly bounded, and for some constants $M,B,\rho$ and any $t\in[0,\tau_1]$,
\begin{align*}
|\sigma_t|\le M(\rho^2+|z_t|^2)^{1/2},~~|b_t|\le B(\rho^2+|z_t|^2)^{1/2}.
\end{align*}
Then for any $\delta>0$ and any $\ep\le1$,
\begin{align*}
\ep\log\P(\sup_{t\in[0,\tau_1]}|z_t|\ge\delta)\le K+\log\Big(\frac{\rho^2+|z_0|^2}{\rho^2+\delta^2}\Big),
\end{align*}
where $K=2B+M^2(2+d)$.
\end{lem}

{\bf Proof of Lemma \ref{equivalence}}
\begin{proof}
Without loss generality, we may choose $R>0$ such that the initial data $x$ is in the ball $B_{R+1}(0)$(center $0$ and radius $R+1$).
We also assume that $X_t^0$ do not leave this ball up to time $T$. We define the stopping time $\tau_R':=\inf\Big\{t:t\ge0\Big||\bar{X}_t^\ep|\vee|\bar{Y}_t^\ep|\ge R+1\Big\}$,
then we denote by $\tau_R=\min\{T,\tau_R'\}$.

In the sequel, we consider $\bar z_t:=\bar{X}_t^\ep-\bar{Y}_t^\ep$, the new process satisfies the following equation
\begin{align}
\d \bar z_t=\int_0^t b_s\d s+\ss\ep\int_0^t\sigma_s\d W_s,
\end{align} 
with 
\begin{align*}
b_t:=\frac{b_t(X_t^\ep,\mathscr{L}_{X_t^\ep})-b_t(Y_t^\ep,\delta_{X_t^0})}{\ss\ep\lambda(\ep)},~~~
\sigma_t:=\frac{\sigma_t(X_t^\ep,\mathscr{L}_{X_t^\ep})-\sigma_t(Y_t^\ep,\delta_{X_t^0})}{\ss\ep\lambda(\ep)}.
\end{align*}
Notice that, both $b_t$ and $\sigma_t$ are progressively measurable process , we now assume $t\le\tau_R$.
\begin{align*}
|b_t|&=\frac{|b_t(X_t^\ep,\mathscr{L}_{X_t^\ep})-b_t(X_t^\ep,\delta_{X_t^0})+b_t(X_t^\ep,\delta_{X_t^0})-b_t(Y_t^\ep,\delta_{X_t^0})|}{\ss\ep\lambda(\ep)}\\
&\le\frac{K(t)W_2(\mathscr{L}_{X_t^\ep},\delta_{X_t^0})}{\ss\ep\lambda(\ep)}+\frac{K(t)|X_t^\ep-Y_t^\ep|}{\ss\ep\lambda(\ep)}\\
&\le K(t)(\rho^2(\ep)+|\bar z_t|^2)^{1/2},
\end{align*}
where $\rho^2(\ep)=\sup_{0\le t\le T}\E|\bar{X}_t^\ep|^2$.
In the same vein, we have
\begin{align*}
|\sigma_t|\le K(t)(\rho^2(\ep)+|\bar z_t|^2)^{1/2}.
\end{align*}

Notice that $\bar{X}_0^\ep=\bar{Y}_0^\ep$, by the Lemma \ref{b}, we have for any $\delta,\rho^\ep$ and for any $\ep$ small enough, we have
\begin{align*}
\ep\log\P\Big(\sup_{t\in[0,\tau_R]}|\bar z_t|\ge\delta\Big)\le KT+\log\Big(\frac{\rho^2(\ep)}{\rho^2(\ep)+\delta^2}\Big).
\end{align*}
As $\rho^2(\ep)$ converges to $0$, as $\ep\rightarrow0$, we deduce that
\begin{align*}
\limsup_{\ep\rightarrow0}\ep\log\P\Big(\sup_{t\in[0,\tau_R]}|\bar z_t|\ge\delta\Big)=-\8.
\end{align*}
Now, since
\begin{align*}
\{|\bar{X}^\ep-\bar{Y}^\ep|_\8\ge\delta\}\subset\{\tau_R\le T\}\cup\Big\{\sup_{0\le t\le \tau_R}|\bar{X}_t^\ep-\bar{Y}_t^\ep|\ge\delta\Big\},
\end{align*}
we can conclude as long as we show that
\begin{align*}
\lim_{R\rightarrow\8}\limsup_{\ep\rightarrow0}\ep\log\Big(\P\{\tau_R<T\}\Big)=-\8.
\end{align*}
By $\eta_R$, we denote the first time $\bar{Y}^\ep$ exits from the ball $B_R(0)$ ( center $0$ and radius $R$). If $\bar{Y}_{\tau_R}^\ep$ is not in the ball $B_{R+1}(0)$, then we have immediately $\eta_R<T$.
If $\bar{X}_t^\ep$ is not in the ball $B_{R+1}(0)$, by taking $\delta<\frac{1}{2}$, we know that with probability $\bar{Y}_{\tau_R}^\ep$ is not in the ball $B_R(0)$, which means $\eta_R<T$.  Therefore we have $\P\{\tau_R<T\}=\P\{\eta_R<T\}.$
That is to say, to end the proof, it is sufficient to prove that the probability that $\bar{Y}^\ep$ exits the ball  $B_R(0)$ is very small as $\ep$ goes to zero, i.e.
\begin{align*}
\lim_{R\rightarrow\8}\limsup_{\ep\rightarrow0}\ep\log\Big(\P\{\eta_R<T\}\Big)=-\infty.
\end{align*}
Recall that $\bar{Y}^\ep$ satisfies an LDP for the uniform norm with good rate function $I(g)$ given in \eqref{rate}. Then, for any closed set $F\subset C([0,T];\R^d)$ we have 
\begin{align*}
\limsup_{\ep\rightarrow0}\ep\log\P\{\bar{Y}^\ep\in F\}\le -\inf_{g\in F}I(g).
\end{align*}
As a consequence, 
\begin{align*}
\limsup_{\ep\rightarrow0}\ep\log\Big(\P\{\tau_R<T\}\Big)
&=\limsup_{\ep\rightarrow0}\ep\log\Big(\P\{\eta_R<T\}\Big)\\
&=\limsup_{\ep\rightarrow0}\ep\log\Big(\P\Big\{\sup_{0\le t\le T}|\bar{Y}_t^\ep|\ge R\Big\}\Big)\\
&\le-\inf_{\{h\in\H;g=\Gamma^0(\int_0^\cdot\dot{h}(s)\d s),\|g\|_\8\ge R\}}\frac{1}{2}\int_0^T|\dot{h}(s)|^2\d s.
\end{align*}
We remark that the infimum of $I(g)$ on the set of paths exiting from the ball $B_R(0)$ goes to infinity as $R$ goes to infinity. Noting that \begin{align*}
|g(t)|&\le\int_0^t|\nabla_{g(s)}b_s(\cdot,\delta_{X_t^0})(X_t^0)+\sigma_s(X_s^0,\delta_{X_s^0})\dot{h}(s)|\d s\\
&\le\int_0^t|K(s)g(s)|\d s+\int_0^t|K(s)\dot{h}(s)|\d s\\
&\le C_t\Big(\int_0^t|g(s)|^2\d s\Big)^{1/2}+C_t\Big(\int_0^t|\dot{h}(s)|^2\d s\Big)^{1/2},
\end{align*} 
and  the Gronwall inequality, we have 
\begin{align*}
|g(t)|^2\le C_t\int_0^t|\dot{h}(s)|^2\d s,
\end{align*}
the desired assertion arrived at by taking $R\rightarrow\8$.
\end{proof}

\end{document}